\documentclass[10pt,a4paper,reqno]{amsart}
\usepackage{epsfig}
\usepackage{color}
\usepackage{amssymb,amsmath,amsthm,amstext,amsfonts}
\usepackage{amssymb,amsmath,amsthm,amstext,amsfonts}
\usepackage{amsmath,amstext,amsthm,amsfonts}
\usepackage{color}
\usepackage[mathscr]{eucal}
\usepackage{dsfont}

\usepackage{enumitem}

\usepackage{hyperref}
\RequirePackage[dvipsnames]{xcolor} 
\definecolor{halfgray}{gray}{0.55} 
\definecolor{webgreen}{rgb}{0,0.5,0}
\definecolor{webbrown}{rgb}{.6,0,0} \hypersetup{%
	colorlinks=true, linktocpage=true, pdfstartpage=3,
	pdfstartview=FitV,%
	breaklinks=true, pdfpagemode=UseNone, pageanchor=true,
	pdfpagemode=UseOutlines,%
	plainpages=false, bookmarksnumbered, bookmarksopen=true,
	bookmarksopenlevel=1,%
	hypertexnames=true,
	pdfhighlight=/O,
	urlcolor=webbrown, linkcolor=RoyalBlue,
	citecolor=webgreen, 
	pdftitle={},%
	pdfauthor={},%
	pdfsubject={2000 MAthematical Subject Classification: Primary:},%
	pdfkeywords={},%
	pdfcreator={pdfLaTeX},%
	pdfproducer={LaTeX with hyperref}%
}

\usepackage{psfrag}
\usepackage{url}
\usepackage{epstopdf}
\usepackage{epic,eepic}
\usepackage{upgreek}
\usepackage{amssymb}
\usepackage{mathrsfs}
\usepackage{verbatim}
\usepackage{mathrsfs}
\usepackage{amstext}
\usepackage{amsthm}
\usepackage{amssymb}
\usepackage{graphicx}

\usepackage[]{mdframed}

\usepackage{hyperref}
\RequirePackage[dvipsnames]{xcolor} 
\definecolor{halfgray}{gray}{0.55} 
\definecolor{webgreen}{rgb}{0,0.5,0}
\definecolor{webbrown}{rgb}{.6,0,0} \hypersetup{%
	colorlinks=true, linktocpage=true, pdfstartpage=3,
	pdfstartview=FitV,%
	breaklinks=true, pdfpagemode=UseNone, pageanchor=true,
	pdfpagemode=UseOutlines,%
	plainpages=false, bookmarksnumbered, bookmarksopen=true,
	bookmarksopenlevel=1,%
	hypertexnames=true,
	pdfhighlight=/O,
	urlcolor=webbrown, linkcolor=RoyalBlue,
	citecolor=webgreen, 
	pdftitle={},%
	pdfauthor={},%
	pdfsubject={2000 MAthematical Subject Classification: Primary:},%
	pdfkeywords={},%
	pdfcreator={pdfLaTeX},%
	pdfproducer={LaTeX with hyperref}%
}

  [1995/08/01 v0.1 Double stroke roman fonts]

\def\ds@whichfont{dsrom}
\relax

\DeclareMathAlphabet{\mathds}{U}{\ds@whichfont}{m}{n}

\newtheorem{theorem}{Theorem}[section]
\newtheorem{lemma}[theorem]{Lemma}

\newtheorem{example}[theorem]{Example}

\theoremstyle{definition}
\newtheorem{definition}[theorem]{Definition}

\newtheorem{remark}[theorem]{Remark}

\theoremstyle{plain}

\theoremstyle{plain}
\theoremstyle{plain}
\theoremstyle{remark}
\newtheorem*{acknowledgement*}{Acknowledgement}

\DeclareMathOperator{\Ker}{Ker}
\newcommand{\cA}{{\mathcal A}}

\newcommand{\bbN}{{\mathbb N}}

\newcommand{\bbR}{{\mathbb R}}

\newcommand{\bbZ}{{\mathbb Z}}

\newcommand{\Id}{\text{Id}}

\usepackage{orcidlink}

\begin{document}
\title[A characterization of $(\mu,\nu)$-dichotomies via admissibility]{A characterization of $(\mu,\nu)$-dichotomies via admissibility}

\author[Lucas Backes]{Lucas Backes \orcidlink{0000-0003-3275-1311}}
\address{\noindent Departamento de Matem\'atica, Universidade Federal do Rio Grande do Sul, Av. Bento Gon\c{c}alves 9500, CEP 91509-900, Porto Alegre, RS, Brazil.}
\email[Lucas Backes]{lucas.backes@ufrgs.br}

\author[Davor Dragi\v cevi\' c]{Davor Dragi\v cevi\' c \orcidlink{0000-0002-1979-4344}} 
\address{Department of Mathematics, University of Rijeka, Croatia}
\email[Davor Dragi\v cevi\' c]{ddragicevic@math.uniri.hr}


\maketitle

\begin{abstract}
We present a characterization of $(\mu,\nu)$-dichotomies in terms of the admissibility of certain pairs of weighted spaces for nonautonomous discrete time dynamics acting on Banach spaces. Our general framework enables us to treat various settings in which no similar result has been previously obtained as well as to recover and refine several known results. We emphasize that our results hold without any bounded growth assumption and the statements make no use of Lyapunov norms.
Moreover, as a consequence of our characterization, we study the robustness  of $(\mu, \nu)$-dichotomies, i.e. we show that this notion persists under small but very general linear perturbations. 
\end{abstract}

\section{Introduction}

Hyperbolicity is one of the cornerstone notions in the modern theory of dynamical systems. 
Roughly speaking, a system is said to be hyperbolic if the phase space  splits into two complementary directions such that along one of these directions we have exponential expansion with time, while in the other one we have exponential contraction. On the other hand, sometimes it may be very complicated to verify if a given system is hyperbolic. Consequently, an important problem consists in presenting different characterizations of this property. 

In the present paper we are interested in characterizing the nonautonomous version of hyperbolicity. More precisely, we will present a characterization of  the notion of \emph{$(\mu,\nu)$-dichotomy}.  
The notion of $(\mu, \nu)$-dichotomy requires that the phase space splits (at each moment of time) into two directions: the stable and the unstable direction. Along stable/unstable direction dynamics contracts/expands with the rate of contraction/expansion  given by a function $\mu$, while the speed of contraction/expansion is measured using a function $\nu$. We stress that the notion of $(\mu, \nu)$-dichotomy includes the notion of (nonuniform) exponential dichotomy as a very particular case.
The type of characterization we are looking for has a long history, which goes back to the work of Perron \cite{Per} for ODEs, and is given in terms of what is nowadays called the \emph{admissibility property}. 

Given a nonautonomous dynamical system
\begin{equation}\label{eq: LE intro}
x_{n+1}=A_nx_n,\quad n\in I,
\end{equation}
where $A_n\colon X\to X$, $n\in I$, are linear maps acting on a Banach space $(X,\|\cdot \|)$ and $I$ is an interval of $\bbZ$, we say that the pair $(Y, Z)$ is \emph{(properly) admissible} for Eq.~\eqref{eq: LE intro}, where $Y$ and $Z$ are subspaces of $X^{I}$, if for every sequence $(y_n)_{n\in I}$ in $Y$ there exists a (unique) sequence $(x_n)_{n\in I}$ in $Z$ such that
\[
    x_{n+1}=A_nx_n+y_{n+1}, \quad \text{ for all } n\in I.
\]
Thus, a prototype result says that for $I=\bbZ$, if $Y=Z=\ell^\infty$ (the space of all bounded two-sided sequences in $X$) then the proper admissibility of the pair $(Y,Z)$ is equivalent to the fact that Eq.~\eqref{eq: LE intro} admits an exponential dichotomy (see \cite{Hen}) and the characterizations of $(\mu,\nu)$-dichotomy that we are seeking have this  flavor. Nevertheless, our results involve two or three (depending on $I$) admissible pairs of some special \emph{weighted spaces} instead of just $\ell^\infty$ (see Sections \ref{sec: one sided} and \ref{sec: two sided}). 

The importance of our results stems from our general framework. More precisely, we are able to treat in a unified manner various settings in which no similar result has been previously obtained as well as to recover and refine several known results. We emphasize that our results hold without any bounded growth assumption for \eqref{eq: LE intro} and the statements make no use of Lyapunov norms.
We explain a bit more the importance of these facts bellow. Moreover, as a consequence of our characterization, we obtain that the notion of $(\mu, \nu)$-dichotomy persists under small but very general linear perturbations.

\subsection{Relations with previous results} 
As already emphasized, characterizations of dichotomies in terms of admissibility properties have a long history. For instance, a fundamental contribution to this line of research is due to Massera and Sch\"affer \cite{MS1,MS2} (see also Coppel \cite{Cop}) who presented a complete characterization (in terms of admissibility) of the notion of an exponential dichotomy, extending the original work of Perron \cite{Per} (see also~\cite{Li} for related results in the case of discrete time) which dealt with exponentially stable systems. 
The case of infinite-dimensional dynamics were first considered by Dalec$'$ki\u{\i} and  Kre\u\i{n} \cite{DK} in the case of continuous time and by Coffman and Sch\"{a}ffer~\cite{CS} as well as Henry \cite{Hen} in the case of noninvertible discrete time dynamics. More recent contributions devoted to the characterization of uniform exponential dichotomies 
include (but are not limited to) \cite{AM, Huy, LRS, MSS2, MRS, S, SS2, SS1}. We also refer to~\cite{BDV} for a detailed overview of this line of research and additional references.

In the case of \emph{nonuniform} exponential dichotomies, there are essentially three type of results available. Firstly, there are results which do not give a complete characterization of nonuniform exponential dichotomies  via admissibility but rather only sufficient conditions for the existence of dichotomy (see~\cite{MSS1, PM, SBS, SS} and references therein). Secondly, there are various results in which a complete characterization of nonuniform dichotomies via admissibility is obtained in which the corresponding input-output spaces are constructed in terms of Lyapunov norms which are used to transform nonuniform behavior into the uniform one (see~\cite{BDV0, BDV1, BDV2, LP1, LP2,  ZLZ}). Finally, since  Lyapunov norms are difficult to construct without firstly  establishing the existence of nonuniform behavior, it was of interest to explore whether nonuniform exponential dichotomies can be characterized in terms of admissibility without the use of Lyapunov norms. The number of very recent results show that this is indeed possible (see~\cite{DZZ, DZZ1, WX, ZZ}).

As already mentioned, besides exponential dichotomies, it of interest to study dichotomic behavior when the rates of expansion/contraction along the unstable/stable direction are not of exponential type. To the best of our knowledge such dichotomies were first studied by Muldowney~\cite{M} and Naulin and Pinto~\cite{NP}. More recently, a systematic study of such behavior was initiated by Barreira and Valls~\cite{BV}, as well as Bento and Silva~\cite{BS2, BS1}. Dragi\v cevi\' c~\cite{D1} characterized nonuniform polynomial dichotomies in terms of admissibility for discrete dynamics (see~\cite{D2} for related result for continuous time). An alternative approach which relies on the relationship between exponential and polynomial dichotomies was proposed in~\cite{DSS}. In the case of continuous time, a very general class of dichotomies associated to differentiable growth rates was characterized via admissibility in~\cite{DPL}.
Finally,  the most general results in the case of discrete time were obtained  by Silva \cite{Silva}. We stress that all these works rely on the usage of Lyapunov norms. 

In the present paper we aim to characterize $(\mu, \nu)$-dichotomies in terms of  admissibility without the use of Lyapunov norms. Therefore, our results complement those obtained by Silva~\cite{Silva}. We emphasize that even in the case of uniform dichotomies associated to a growth rate $\mu$, our results do not coincide with those in~\cite{Silva} as the input-output spaces are different. Moreover, in contrast to~\cite{Silva} we:
\begin{enumerate}
\item do not impose any bounded growth conditions;
\item are able to treat the case of arbitrary growth rates, while in~\cite{Silva} it is required that $\mu$ is ``slowly growing'';
\item discuss the case of two-sided dynamics which was not treated in~\cite{Silva}.
\end{enumerate}
Our proof was inspired by the work of Dragi\v cevi\'c, Zhang and Zhou \cite{DZZ, DZZ1} in which the authors completely characterize, in terms of admissibility, the notion of \emph{nonuniform exponential dichotomy}, which, as already mentioned, is a particular example of $(\mu,\nu)$-dichotomy.

The paper is organized as follows. In Section \ref{sec: prelim}, we introduce the notion of $(\mu,\nu)$-dichotomy and the weighted spaces we are going to work with along the text. Section \ref{sec: one sided} is devoted to present the characterization as well as the robustness of $(\mu,\nu)$-dichotomies in the case of one-sided dynamics while Section \ref{sec: two sided} is devoted to study the two-sided case.

\section{Preliminaries}\label{sec: prelim}

Let $X=(X, \|\cdot \|)$ be an arbitrary Banach space and $I$ be either equal to $\bbZ$ or $\bbN$. By $\mathcal B(X)$ we will denote the space of all bounded linear operators on $X$. The operator norm on $\mathcal B(X)$ will be denoted also by $\|\cdot \|$.  Given a sequence $(A_n)_{n\in I}$ of bounded linear operators in $\mathcal B(X)$, let us consider the associated linear difference equation
\begin{equation}\label{LE}
    x_{n+1}=A_nx_n, \quad n\in I.
\end{equation}
For $m, n\in I$, the evolution operator associated to \eqref{LE} is given by
\[
\cA (m, n)=\begin{cases}
A_{m-1}\cdots A_n & \text{for $m>n$;}\\
\Id  &\text{for $m=n$,} \\
\end{cases}
\]
where $\Id$ denotes the identity operator on $X$.

\subsection{Growth rates and $(\mu,\nu)$-dichotomy}
Let $ \mu =(\mu_n)_{n\in I}$ be a strictly  increasing sequence of positive numbers  such that 
\begin{equation}\label{eq: growth}
     \lim_{n\to +\infty}\mu_n=+\infty.\end{equation}
Moreover, in the case when $I=\bbZ$ we assume also that $\lim\limits_{n\to -\infty}\mu_n=0$. We call such sequence $\mu$ a \emph{growth rate}. Furthermore, let $ \nu =(\nu_n)_{n\in I}$ be an arbitrary sequence with $\nu _n \geq 1$ for every $n\in I$.

\begin{definition}
We say that \eqref{LE} admits a $(\mu,\nu)$-dichotomy if the following conditions are satisfied:
\begin{enumerate}
\item  there exists a family of projections $P_n$, $n\in I$, such that 
\begin{equation}\label{eq: Pn An commute}
A_nP_n=P_{n+1}A_n;
\end{equation}
\item $A_n\rvert_{\Ker P_n} \colon \Ker P_n \to \Ker P_{n+1}$ is an invertible operator for each $n\in I$;
\item there exist $D, \lambda >0$   such that 
\begin{equation}\label{eq: Es discrete}
\|\cA(m,n)P_n\|\leq D\nu_n \left(\frac{\mu_m}{\mu_n}\right)^{-\lambda} \quad \text{ for } m\geq n
\end{equation}
and
\begin{equation}\label{eq: Eu discrete}
\|\cA(m,n)(\Id-P_n)\|\leq D\nu_n \left(\frac{\mu_n}{\mu_m}\right)^{-\lambda} \quad \text{ for } m\leq n
\end{equation}
where 
\[
\cA(m, n):=\big{(}\cA(n, m)\rvert_{\Ker P_m} \big{)}^{-1} \colon \Ker P_n \to \Ker P_m,
\]
for $m\le n$. 

\end{enumerate}
\end{definition}

\begin{remark}\label{remark: particular cases of dich}
We observe that the notion of $(\mu, \nu)$-dichotomy, which appeared earlier, for instance, in \cite{BS2, BS1, Pinto}, generalizes several well-known notions of dichotomy. To illustrate this claim, let us restrict ourselves to the case when $I=\bbN$ and suppose initially that $\nu_n=C$ for some $C\geq 1$ and $n\in \bbN$. Then, by taking $\mu_n=e^{n}$, $n\in \bbN$, we recover the notion of \emph{exponential dichotomy} (see \cite{Cop,MS1,MS2}); by taking $\mu_n=1+n$, $n\in \bbN$, we recover the notion of \emph{polynomial dichotomy} (see \cite{D1}); by taking $\mu_n= \log(2+n)$, $n\in \bbN$, we recover the notion of \emph{logarithmic dichotomy} (see \cite{Silva}). In all these cases, if we take $\nu_n=\mu_n^\varepsilon$ instead of $\nu_n=C$ for some small $\varepsilon>0$ and $n\in \bbN$, we get \emph{nonuniform} versions of those dichotomies.
\end{remark}

\subsection{Weighted spaces} In order to describe our main results, we need to introduce a number of special \emph{weighted spaces}. 
Given a growth rate $\mu=(\mu_n)_{n\in I}$ and $\beta \in \bbR$, let us consider the weighted space $\ell_\beta^\infty $ which consist of all sequences $\mathbf x=(x_n)_{n\in I}\subset X$ such that 
\[
\|\mathbf x\|_{\infty, \beta}:=\sup_{n\in I}(\mu_n^\beta \|x_n\|)<+\infty.
\]
Then, it is not difficult to check that $\|\cdot\|_{\infty, \beta}$ is a norm in $\ell_\beta ^\infty $ and $(\ell_\beta^\infty, \| \cdot \|_{\infty, \beta})$ is a Banach space. Similarly, if $\nu=(\nu_n)_{n\in I}$ is a sequence such that $\nu_n\ge 1$ for each $n\in I$, we consider the weighted space $\ell^1_\beta$ which consist of all sequences $\mathbf x=(x_n)_{n\in I}\subset X$ such that 
\[
\| \mathbf x\|_{1, \beta}:=\sum_{n\in I}  \mu_n^\beta  \nu_n \|x_n\|<+\infty.
\]
Again it is not difficult to verify that $(\ell_\beta^1, \| \cdot \|_{1, \beta})$ is a Banach space. Moreover, given a closed subspace $Z\subset X$, for $j\in \{1,\infty\}$, let 
$\ell_{\beta, Z}^j$ be the space which consists of all sequences $\mathbf x=(x_n)_{n\in I}\in \ell_\beta^j$ such that $x_0\in Z$. Then, 
$\ell_{\beta, Z}^j$ is a closed subspace of $\ell_\beta^j$ and, in particular, $(\ell_{\beta,Z}^j, \| \cdot \|_{j, \beta})$ is a Banach space. In the particular case when $Z=\{0\}$, $\ell^1_{\beta, Z}$ will be denoted by $\ell^1_{\beta, 0}$.

In the case when $I=\bbZ$, we will also need to consider some extra spaces described as follows. Observe initially that, since $(\mu_n)_{n\in \bbZ}$ is strictly increasing and $\lim_{n\to -\infty}\mu_n=0$ and $\lim_{n\to +\infty}\mu_n=+\infty$, there exists $n_0\in \bbZ$ such that $\mu_n< 1$ for every $n< n_0$ and $\mu_n\geq 1$ for every $n\geq n_0$. Then, we consider the space $\ell_{\beta,|\cdot |}^\infty $ which consist of all sequences $\mathbf x=(x_n)_{n\in \mathbb Z}\subset X$ such that 
\[
\|\mathbf x\|_{\infty, \beta,|\cdot |}:=\max \left\{\sup_{n< n_0}(\mu_n^{|\beta|} \|x_n\|), \sup_{n\geq n_0}(\mu_n^{-|\beta|} \|x_n\|)\right\} <+\infty.
\]
Similarly, let $\ell^1_{\beta,|\cdot |}$ be the space consisting of all sequences $\mathbf x=(x_n)_{n\in \mathbb Z}\subset X$ such that 
\[
\| \mathbf x\|_{1, \beta,|\cdot |}:=\sum_{n<n_0 } \mu_n^{|\beta|}  \nu_n \|x_n\|+\sum_{n\geq n_0 } \mu_n^{-|\beta|}  \nu_n \|x_n\|<+\infty.
\]
Once again one can easily check that $(\ell_{\beta,|\cdot|}^\infty, \| \cdot \|_{\infty, \beta,|\cdot|})$ and $(\ell_{\beta,|\cdot|}^1, \| \cdot \|_{1, \beta,|\cdot|})$ are Banach spaces. Observe that, even though all the spaces defined above depend on $\mu$, $\nu$ and $I$, we do not write these dependence explicitly in order to simplify notation. The dependence will be clear from the context.

\section{The case of one-sided dynamics}\label{sec: one sided}
In this section we will restrict our attention to the case of one-sided dynamics, that is, to the case when $I=\bbN$. 

\subsection{Characterization of $(\mu,\nu)$-dichotomy} \label{sec: charct one sided}
We will now present a characterization of $(\mu ,\nu)$-dichotomy in terms of the admissibility of the spaces $\ell^1_{\beta,0}$ and $\ell^\infty_{\beta,Z}$ for some appropriate subspace $Z\subset X$ and $\beta\in \bbR$.

\begin{theorem}\label{theo: from dich to adm}
Suppose that $(A_n)_{n\in \mathbb N}$ admits a $(\mu,\nu)$-dichotomy with respect to projections $P_n$ and let $\lambda>0$ be such that~\eqref{eq: Es discrete} and~\eqref{eq: Eu discrete} hold. Moreover, suppose that there exists $\varepsilon \in [0, \lambda)$ such that \begin{equation}\label{munu}\sup_{n\in \mathbb N}(\mu_n^{-\varepsilon}\nu_n)<+\infty.
\end{equation}
Set $Z:=\Ker P_0$. Then,  for each $\beta \in (-(\lambda-\varepsilon), \lambda)$ and $\mathbf y=(y_n)_{n\in \mathbb N}\in \ell^1_{\beta, 0}$, there exists a unique $\mathbf x=(x_n)_{n\in \mathbb N}\in \ell_{\beta, Z}^\infty$ such that
\begin{equation}\label{adm}
x_{n+1}-A_n x_n=y_{n+1}, \quad n\in \mathbb N.
\end{equation}
\end{theorem}

\begin{proof}
  Take $\beta \in (-(\lambda-\varepsilon), \lambda)$, $\mathbf y=(y_n)_{n\in \mathbb N}\in \ell^1_{\beta, 0}$  and set
\begin{equation}\label{xn}
x_n:=\sum_{k=0}^n \mathcal A(n, k)P_k y_k-\sum_{k=n+1}^\infty \mathcal A(n, k)(\Id-P_k)y_k, \quad n\in \mathbb N.
\end{equation}
Then,
\[
\begin{split}
\mu_n^\beta \|x_n\| &\le D\mu_n^\beta \sum_{k=0}^n \left (\frac{\mu_n}{\mu_k}\right )^{-\lambda}\nu_k \|y_k\|+D\mu_n^\beta \sum_{k=n+1}^\infty \left (\frac{\mu_k}{\mu_n}\right )^{-\lambda}\nu_k \|y_k\| \\
&\le D\sum_{k=0}^n \left (\frac{\mu_n}{\mu_k}\right )^{-(\lambda-\beta)}\mu_k^\beta \nu_k \|y_k\|+D\sum_{k=n+1}^\infty \left (\frac{\mu_k}{\mu_n}\right )^{-(\lambda+\beta)}\mu_k^\beta \nu_k \|y_k \|, \\
\end{split}
\]
which implies (since $\lambda-\beta>0$ and $\lambda+\beta>0$) that 
\[
\mu_n^{\beta} \|x_n\|\le D\sum_{k=0}^\infty \mu_k^\beta \nu_k \|y_k\|, \quad n\in \mathbb N.
\]
Note also that $x_0\in Z$ and, consequently, $\mathbf x=(x_n)_{n\in \mathbb N}\in \ell^\infty_{\beta, Z}$. Moreover, it is easy to verify that $\mathbf x$ satisfies ~\eqref{adm}. Indeed, given $n\in \bbN$,
\[
\begin{split}
 A_nx_n
   &= A_n\sum_{k=0}^{n}\cA(n,k)P_{k}y_{k} - A_n\sum_{k=n+1}^{\infty}\cA(n,k)(\Id -P_{k})y_{k}
      \\
  &= \sum_{k=0}^{n}\cA(n+1,k)P_{k}y_{k}- \sum_{k=n+1}^{\infty}\cA(n+1,k)(\Id -P_{k})y_{k}
      \\
  &= \sum_{k=0}^{n+1}\cA(n+1,k)P_{k}y_{k} - P_{n+1}y_{n+1}\\
  &\phantom{=} - \sum_{k=n+2}^{\infty}\cA(n+1,k)(\Id -P_{k})y_{k}  -(\Id -P_{n+1})y_{n+1}\\
   &= x_{n+1} -y_{n+1},
\end{split}
\]
as claimed. It remains to establish the uniqueness of such sequence $\mathbf x$. Suppose that $\tilde{\mathbf x}=(\tilde x_n)_{n\in \mathbb N}$ is another sequence in $\ell^\infty_{\beta, Z}$ that satisfies~\eqref{adm}. Then,
\[
x_n-\tilde x_n=\mathcal A(n, 0)(x_0-\tilde x_0), \quad n\in \mathbb N.
\]
Consequently, since $x_0-\tilde x_0\in Z=\Ker P_0$, condition \eqref{eq: Eu discrete} gives us that 
\[
\begin{split}
\mu_0^{-\lambda}\|x_0-\tilde x_0\| &\le D\mu_n^{-\lambda}\nu_n \|x_n-\tilde x_n\| \\
&\le D\mu_n^{-(\lambda+\beta)} \nu_n \|\mathbf x-\tilde{\mathbf x}\|_{\infty, \beta} \\
&\le D\mu_n^{-(\lambda+\beta-\varepsilon)}\|\mathbf x-\tilde{\mathbf x}\|_{\infty, \beta} \cdot \sup_{n\in \mathbb N}(\mu_n^{-\varepsilon}\nu_n),
\end{split}
\]
for $n\in \mathbb N$. Letting $n\to \infty$,  using~\eqref{munu} and $\beta>-(\lambda-\varepsilon)$
yields that $x_0=\tilde x_0$ and thus $\mathbf x=\tilde{\mathbf x}$. The proof of the theorem is completed. 
\end{proof}

We now establish the converse result.
\begin{theorem}\label{theo: adm to dich}
Suppose that there exist a closed subspace $Z\subset X$ and  $\beta>0$  such that:
\begin{enumerate}
\item for each $\mathbf y=(y_n)_{n\in \mathbb N} \in \ell^1_{\beta, 0}$, there exists a unique $\mathbf x=(x_n)_{n\in \mathbb N}\in \ell_{\beta, Z}^\infty$ such that~\eqref{adm} holds;
\item for each $\mathbf y=(y_n)_{n\in \mathbb N}\in \ell^1_{-\beta, 0}$, there exists a unique $\mathbf x=(x_n)_{n\in \mathbb N}\in \ell_{-\beta, Z}^\infty$ such that~\eqref{adm} holds.
\end{enumerate}
Then, $(A_n)_{n\in \mathbb N}$ admits a $(\mu,\nu)$-dichotomy. 
\end{theorem}

\begin{proof} We will split the proof into a series of auxiliary results and, in all of them, we assume that the hypotheses of Theorem \ref{theo: adm to dich} are satisfied even though we do not write it explicitly each time.

\begin{lemma}\label{aux}
Let $\mathbf y=(y_n)_{n\in \mathbb N}\in \ell^1_{\beta, 0}\cap \ell^1_{-\beta, 0}$, $\mathbf x^1=(x_n^1)_{n\in \mathbb N}\in \ell^\infty_{\beta, Z}$ and $\mathbf x^2=(x_n^2)_{n\in \mathbb N}\in \ell^\infty_{-\beta, Z}$ be such that both pairs $(\mathbf x^i, \mathbf y)$, $i\in \{1, 2\}$ satisfy~\eqref{adm}. Then, $x_n^1=x_n^2$ for every $n\in \mathbb N$. 
\end{lemma}
\begin{proof}[Proof of Lemma \ref{aux}]
    The result follows directly from the `uniqueness' in our assumptions and the simple observation that $\ell^\infty_{\beta, Z}\subset \ell^\infty_{-\beta, Z}$.
\end{proof}

Let us consider
\begin{equation}\label{Sn}
S(n):=\left\{v\in X: \sup_{m\ge n}\|\mathcal A(m, n)v\|<+\infty \right\}
\end{equation}
and 
\[
U(n)=\mathcal A(n, 0)Z.
\]
It is easy to check that 
\begin{equation}\label{eq: spliting is invariant}
A_nS(n)\subset S(n+1) \text{ and } A_nU(n)\subset U(n+1)
\end{equation}
 for every $n\in \bbN$.
\begin{lemma}\label{lem: spliting}
We have that 
\begin{equation}\label{split}
X=S(n)\oplus U(n), \quad n\in \mathbb N.
\end{equation}
\end{lemma}
\begin{proof}[Proof of Lemma \ref{lem: spliting}]
Take initially $n\ge 1$. Given $v\in X$, we define $\mathbf y=(y_m)_{m\in \mathbb N}$ by $y_n=v$ and $y_m=0$ for $m\neq n$. Then, $\mathbf y\in \ell^1_{\beta, 0}$. By our assumption, there exists $\mathbf x=(x_m)_{m\in \mathbb N}\in \ell^\infty_{\beta, Z}$ such that 
\[
x_{m+1}-A_m x_m=y_{m+1}, \quad m\in \mathbb N.
\]
In particular, 
\[
x_n-A_{n-1}x_{n-1}=v.
\]
Then, since $x_0\in Z$, we have that $A_{n-1}x_{n-1}=\mathcal A(n, 0)x_0\in U(n)$. Moreover, since $x_m=\mathcal A(m, n)x_n$ for $m\ge n$ and $\mathbf x\in \ell^\infty_{\beta, Z}$, we conclude that  $x_n\in S(n)$. Thus, $v\in S(n)+U(n)$.

Suppose now that $v\in S(n)\cap U(n)$. Then, there exists $w\in Z$ such that $v=\mathcal A(n, 0)w$. We define
\[
x_m:=\mathcal A(m, 0)w, \quad m\in \mathbb N.
\]
Then, $\mathbf x=(x_m)_{m\in \mathbb N}\in \ell^\infty_{-\beta, Z}$ and~\eqref{adm} holds with $\mathbf y=0$. Therefore, using the uniqueness in our assumption, we conclude that $\mathbf x=0$ and $v=0$ proving that~\eqref{split} holds  for every $n\geq 1$.

Let us now consider the case when $n=0$. Given $v\in X$, we define  sequences 
$\mathbf x=(x_m)_{m\in \mathbb N}$ and $\mathbf y=(y_m)_{m\in \mathbb N}$ given by $x_0=v$ and $x_m=0$ for $m\neq n$ and $y_{1}=-A_0v$ and $y_m=0$ for $m\neq 1$. Then,
\[x_{m+1}-A_mx_m=y_{m+1}, \quad m\in \bbN.\]
Moreover, since $\mathbf y \in \ell^1_{\beta, 0}$, there exists $\tilde{\mathbf x}=(\tilde x_m)_{m\in \bbN}\in \ell^\infty_{\beta,Z}$ such that \eqref{adm} holds. Then,
\[x_m-\tilde x_m=\mathcal A (m,0)(v-\tilde x_0), \quad \text{ for all }m\in \bbN.\]
Consequently, since $\mathbf x-\tilde{\mathbf x}\in \ell^\infty_{\beta}$, it follows that $v-\tilde x_0\in S(0)$. Therefore, since $\tilde x_0\in Z$, we conclude that $v\in S(0)+U(0)$. Finally, consider $v\in S(0)\cap U(0)$ and let $x_m=\mathcal A(m,0)v$ for $m\in \bbN$.
Then, as in the case when $n\geq 1$, $\mathbf x=(x_m)_{m\in \mathbb N}\in \ell^\infty_{-\beta, Z}$ and~\eqref{adm} holds with $\mathbf y=0$. Therefore, using the uniqueness in our assumption, we conclude that $\mathbf x=0$ and $v=0$. This completes the proof of the lemma.
\end{proof}

\begin{lemma}\label{lem: isom}
     For every $n\in \bbN$, the operator $A_n|_{U(n)}\colon U(n)\to U(n+1)$ is an isomorphism.
\end{lemma}
\begin{proof}[Proof of Lemma \ref{lem: isom}]
Fix $n\in \bbN$. Let us start proving that $A_n|_{U(n)}$ is injective. Suppose there exists $v\in U(n)$ such that $A_nv=0$. By the definition of $U(n)$ there exists $z_0\in Z$ such that $v=\mathcal A(n,0)z_0$. Then, the sequence $(x_m)_{m\in \bbN}$ given by 
\[x_m=\mathcal A(m,0)z_0,\quad m\in \bbN,\]
belongs to $\ell^\infty_{\beta, Z}$ and satisfies \eqref{adm} with $\mathbf y=0$. Consequently, by the uniqueness in our hypothesis, it follows that $x_m=0$ for every $m\in \bbN$. In particular, $v=0$ and $A_n|_{U(n)}$ is injective.

Now, given $x\in U(n+1)$, let $z_0\in Z$ be such that $x=\mathcal A(n+1,0)z_{0}$. Set $x':=\mathcal A(n, 0)z_0$. Then, it follows directly from the definition that $x'\in U(n)$ and $A_nx'=x$ which proves that  $A_n|_{U(n)}\colon U(n)\to U(n+1)$ is surjective. Consequently,  $A_n|_{U(n)}\colon U(n)\to U(n+1)$ is an isomorphism as claimed.
\end{proof}

\begin{lemma}\label{lem: projections are bounded}
For each $n\in \bbN$, let $P_n\colon X\to S(n)$ be the projection associated with~\eqref{split}. There exists $D>0$ such that 
\begin{equation}\label{probound}
\sup_{n\in \mathbb N}\|P_n\| \le D\nu_n, \quad n\in \mathbb N.
\end{equation}
\end{lemma}
\begin{proof}[Proof of Lemma \ref{lem: projections are bounded}]
    Let us consider  the map $T_\beta \colon \ell^1_{\beta, 0}\to \ell^\infty_{\beta, Z}$ given by $T_\beta (\mathbf y)=\mathbf x$ where $\mathbf x$ is the unique element in $l^\infty_{\beta,Z}$ such that~\eqref{adm} holds. Then, one can easily check that $T_\beta $ is a linear operator. Moreover, we observe that $T_\beta$ is a closed operator. In fact, let $\mathbf y^k=(y^k_n)_{n\in \bbN}$ be a sequence of elements in $\ell^1_{\beta,Z_0}$ such that $\mathbf y^k\to \mathbf y$ for some $\mathbf y =(y_n)_{n\in \bbN}\in \ell^1_{\beta,Z_0}$ and $T_\beta (\mathbf y^k)\to \mathbf x$ for some $\mathbf x=(x_n)_{n\in \bbN}\in \ell^\infty_{\beta,Z}$. Then, writing $T_\beta (\mathbf y^k)=\mathbf x^k=(x^k_n)_{n\in \bbN}$ we get that
    \[x_{n+1}^k-A_nx^k_n=y^k_{n+1}, \text{ for every } n,k\in \bbN. \]
    Thus, letting $k\to +\infty$ we conclude that
    \[x_{n+1}-A_nx_n=y_{n+1},  \text{ for every } n\in \bbN \]
    which implies that $T_\beta (\mathbf y)=\mathbf x$. In particular, $T_\beta $ is a closed operator and thus, according to the Closed Graph Theorem (see, e.g., \cite[Theorem~4.2-I, p.~181]{Taylor}), $T_\beta $ is bounded.

Now, given $n\ge 1$ and $v\in X$, consider $\mathbf y=(y_m)_{m\in \mathbb N}$ given by $y_n=v$ and $y_m=0$ for $m\neq n$. Then, $\mathbf y\in \ell^1_{\beta, 0}$. Considering $T_\beta (\mathbf y)=\mathbf x=(x_n)_{n\in \bbN}$, it follows by the proof of Lemma \ref{lem: spliting} that $P_nv=x_n$. Thus,
\[\mu_n^\beta \|P_nv\|=\mu_n^\beta\|x_n\|\leq \|\mathbf x\|_{\infty,\beta}= \|T_\beta (\mathbf y)\|_{\infty,\beta} \leq \|T_\beta \| \cdot \|\mathbf y\|_{1,\beta}=\|T_\beta \|\mu_n^\beta \nu_n \|v\|.\]
This yields that~\eqref{probound} holds with $D=\|T_\beta \|$ for $n\ge 1$. On the other hand, \eqref{probound} holds with $D=\|P_0\|$ for $n=0$ (recall that $\nu_0\geq 1$). We conclude that~\eqref{probound} holds with 
\[
D=\max \left \{ \|T_\beta \|, \|P_0\| \right \}>0.
\]
\end{proof}

\begin{lemma}\label{lem: bounds stable direction}
There exists $C>0$ such that 
\[
\|\mathcal A(m, n)v\| \le C\left (\frac{\mu_m}{\mu_n}\right )^{-\beta} \nu_n \|v\|,\]
for every $m\ge n$ and $v\in S(n)$.
\end{lemma}

\begin{proof}[Proof of Lemma \ref{lem: bounds stable direction}]
Let $n\ge 1$ and $v\in S(n)$. We define  sequences $(y_m)_{m\in \mathbb N}$ and $(x_m)_{m\in \mathbb N}$ by 
\[
y_m=\begin{cases}
v & m=n; \\
0 & m\neq n,
\end{cases} \quad \text{and} \quad 
x_m=\begin{cases}
\mathcal A(m, n)v & m\ge n; \\
0& m<n.
\end{cases}
\]
Then, $\mathbf y=(y_m)_{m\in \mathbb N}\in \ell^1_\beta\cap \ell^1_{-\beta}$ and $\mathbf x=(x_m)_{m\in \mathbb Z}\in \ell^\infty_{-\beta, Z}$. Moreover, \eqref{adm} holds. Then, from Lemma~\ref{aux} it follows that $\mathbf x\in \ell^\infty_{\beta, Z}$. Thus, letting $T_\beta \colon \ell^1_{\beta,0}\to \ell^\infty_{\beta, Z}$ be the linear operator defined in the proof of Lemma \ref{lem: projections are bounded} and considering $K=\|T_\beta \|+1$, we get that
\[
\mu_m^\beta \|x_m\| \le \| \mathbf x\|_{\infty, \beta}= \|T_\beta (\mathbf y)\|_{\infty, \beta}\le K\|y\|_{1, \beta}=K\mu_n^\beta \nu_n \|v\|
\]
for $m\ge n$. Thus,
\[
\|\mathcal A(m, n)v\| \le K\left (\frac{\mu_m}{\mu_n}\right )^{-\beta}\nu_n \|v\|
\]
for every $m\ge n$ and $n\geq 1$. 

Now, for the case when $n=0$, if $m=n=0$ then the same estimate as before holds since $K\ge 1$. Take $m\geq 1$. Then, given $v\in S(0)$, by \eqref{eq: spliting is invariant} we have that $A_0v\in S(1)$. Thus, using the estimate obtained above we get that
\[\begin{split}
\|\mathcal A(m, 0)v\|&=\|\mathcal A(m,1)A_0v \| \le K\left (\frac{\mu_m}{\mu_1}\right )^{-\beta} \nu_1 \|A_0v\| \\
&\leq K\|A_0\| \frac{\nu_1}{\nu_0}
\left (\frac{\mu_0}{\mu_1}\right )^{-\beta} \left (\frac{\mu_m}{\mu_0}\right )^{-\beta}\nu_0 \|v\| .
\end{split}
\]
Therefore, taking $C=\max\{K,K\nu_1 \nu_0^{-1}\|A_0\|\mu_0^{-\beta}\mu_1^{\beta}\}$ we get the desired result.
\end{proof}

\begin{lemma}\label{lem: bounds unstable direction}
There exists $\tilde C>0$ such that 
\[
\|\mathcal A(m, n)v\| \le \tilde C \left (\frac{\mu_n}{\mu_m}\right )^{-\beta}\nu_n \|v\|,\]
for every $m\leq  n$ and $v\in U(n)$.
\end{lemma}

\begin{proof}[Proof of Lemma \ref{lem: bounds unstable direction}]     
Take $n\ge 1$ and $v\in U(n)$. Then by Lemma \ref{lem: isom} we may consider 
\[
x_m=\begin{cases}
\mathcal A(m, n)v & m<n;\\
0 &m\ge n,
\end{cases} \quad \text{and} \quad y_m=\begin{cases}
 -v & m=n; \\
0 &m\neq n.
\end{cases}
\]
Note that  $\mathbf x=(x_m)_{m\in \mathbb N}\in \ell^\infty_{-\beta, Z}$ and $\mathbf y=(y_m)_{m\in \mathbb N}\in \ell^1_{-\beta, 0}$. Furthermore, \eqref{adm} holds. Let us now consider the map $ T_{-\beta}\colon \ell^1_{-\beta,Z_0}\to \ell^\infty_{-\beta, Z}$ given by $T_{-\beta}(\mathbf y)=\mathbf x$ where $\mathbf x$ is the unique element in $\ell^\infty_{-\beta,Z}$ associated to $\mathbf y$ by our assumption (so that~\eqref{adm} holds). Then, proceeding as in the proof of Lemma \ref{lem: projections are bounded} we can prove that $ T_{-\beta }$ is a bounded linear operator. Thus, taking $\tilde C=\| T_{-\beta }\|+1$ we get that
\[
\mu_m^{-\beta}\| x_m\| \leq \| \mathbf x\|_{\infty, -\beta}= \| T_{-\beta}(\mathbf y)\|_{\infty, -\beta}\le \tilde C\|y\|_{1, -\beta}=\tilde C\mu_n^{-\beta} \nu_n \|v\|,
\]
for every $m<n$. This easily implies that 
\[
\| \mathcal A(m,n)v\| \le \tilde C \left (\frac{\mu_n}{\mu_m}\right )^{-\beta}\nu_n \|v\|,
\]
for every $m<n$ as claimed. Finally, the cases when $m=n$ or $n=0$ trivially holds since $\tilde C \geq 1$.
\end{proof}
For $m, n\in \mathbb N$, set
\[
\mathcal G(m, n):=\begin{cases}
\mathcal A(m, n)P_n & m\ge n; \\
-\mathcal A(m, n)(\Id-P_n) &m<n.
\end{cases}
\]
Then, by Lemmas~\ref{lem: projections are bounded}, \ref{lem: bounds stable direction} and~\ref{lem: bounds unstable direction} there exists $C'>0$ such that 
\begin{equation}\label{G}
\|\mathcal G(m, n)\| \le C'\begin{cases}
\left (\frac{\mu_m}{\mu_n}\right )^{-\beta}\nu_n^2 & m\ge n; \\
\left (\frac{\mu_n}{\mu_m}\right )^{-\beta}\nu_n^2
& m<n.
\end{cases}
\end{equation}
The purpose of the following lemma is to replace $\nu_n^2$ by $\nu_n$ in~\eqref{G}.
\begin{lemma}\label{Green}
There exists $C''>0$ such that 
\[
\|\mathcal G(m, n)\| \le C''\begin{cases}
\left (\frac{\mu_m}{\mu_n}\right )^{-\beta}\nu_n & m\ge n; \\
\left (\frac{\mu_n}{\mu_m}\right )^{-\beta}\nu_n
& m<n.
\end{cases}
\]
\end{lemma}
\begin{proof}[Proof of the Lemma~\ref{Green}]
Take $n\ge 1$ and $v\in X$. We define a sequence $\mathbf y=(y_m)_{m\in \mathbb N}$ by $y_n=v$ and $y_m=0$ for $m\neq n$. Then, $\mathbf y\in \ell_{\beta, 0}^1\cap \ell_{-\beta, 0}^1$. Set
\[
x_m:=\mathcal G(m, n)v, \quad m\in \mathbb N.
\]
It is easy to verify that~\eqref{adm} holds.
We claim that $\mathbf x=(x_m)_{m\in \mathbb N}\in \ell_{\beta, Z}^\infty$. Indeed, for $m\ge n$ we have that 
\[
\|x_m\|=\|\mathcal G(m, n)v\|\le C'\left (\frac{\mu_m}{\mu_n}\right )^{-\beta}\nu_n^2\|v\|,
\]
which implies that 
\[
\sup_{m\in \mathbb N}(\mu_m^\beta \|x_m\|)<+\infty.
\]
Consequently, $\mathbf x=(x_m)_{m\in \mathbb N}\in \ell_{\beta, Z}^\infty$ and $T_{\beta}(\mathbf y)=\mathbf x$, where $T_\beta$ is as in the proof of Lemma~\ref{lem: projections are bounded}. Therefore, for $m\ge n$ we have that 
\[
\mu_m^\beta \|\mathcal G(m, n)v\|=\mu_m^\beta \|x_m\|\le \|\mathbf x\|_{\infty, \beta} \le \|T_\beta\| \cdot \|\mathbf y\|_{1, \beta}=\|T_\beta\| \mu_n^\beta \nu_n \|v\|,
\]
and thus
\[
\| \mathcal G(m, n)\| \le \|T_\beta \| \left(\frac{\mu_m}{\mu_n}\right)^{-\beta}\nu_n, \quad m\ge n\ge 1.
\]
We now consider the case when $n=0$. For $m>0$ we have 
\[
\begin{split}
\|\mathcal G(m, 0)\|=\|\mathcal G(m, 1)A_0\| &\le \|T_\beta \| \cdot \|A_0\| \left (\frac{\mu_m}{\mu_1}\right )^{-\beta}\nu_1  \\
&=\|T_\beta \| \cdot \|A_0\|\frac{\nu_1}{\nu_0}\left (\frac{\mu_0}{\mu_1}\right )^{-\beta} \left (\frac{\mu_m}{\mu_0}\right )^{-\beta}\nu_0.
\end{split}
\]
Hence, taking
\[C''\geq \max \left \{ \|T_\beta \|, \|T_\beta \| \cdot \|A_0\|\frac{\nu_1}{\nu_0}\left (\frac{\mu_0}{\mu_1}\right )^{-\beta}, \|P_0\| \right \}\]
we get that 
\[
\| \mathcal G(m, n)\|\leq C'' \left(\frac{\mu_m}{\mu_n}\right)^{-\beta}\nu_n, 
\]
for $m\ge n$. Similarly, one can treat the case when $m<n$.
\end{proof}

Now the proof of Theorem \ref{theo: adm to dich} can be easily completed by combining the previous auxiliary results. 
\end{proof}

\subsection{An alternative characterization of $(\mu,\nu)$-dichotomy for compact operators}
We now present an alternative to Theorem~\ref{theo: adm to dich} in the case when at least one of the operators $A_n$ is compact. For this purpose, given $\beta \in \mathbb R$, let 
\[
S_\beta(0)=\left \{ v\in X: \ \sup_{n\in \mathbb N}(\mu_n^\beta \|\cA(n, 0)v\|)<+\infty  \right \}.
\]
Note that $S_\beta(0)$ is a subspace of $X$.
\begin{theorem}\label{new}
Suppose that $A_n$ is compact \emph{for some} $n\in \mathbb N$ and that  there exists   $\beta>0$  such that:
\begin{enumerate}
\item $S_0(0)=S_\beta(0)$;
\item for each $\mathbf y=(y_n)_{n\in \mathbb N} \in \ell^1_{\beta, 0}$, there exists  $\mathbf x=(x_n)_{n\in \mathbb N}\in \ell_{\beta}^\infty$ such that~\eqref{adm} holds;
\item for each $\mathbf y=(y_n)_{n\in \mathbb N}\in \ell^1_{-\beta, 0}$, there exists  $\mathbf x=(x_n)_{n\in \mathbb N}\in \ell_{-\beta}^\infty$ such that~\eqref{adm} holds.
\end{enumerate}
Then, $(A_n)_{n\in \mathbb N}$ admits a $(\mu,\nu)$-dichotomy. 
\end{theorem}

\begin{proof}
For $n\in \mathbb N$, let $S(n)$ be given by~\eqref{Sn}. We first claim that
\[
A_n^{-1}(S(n+1))=S(n), \quad  n\in \mathbb N.
\]
To this end, take $v\in A_n^{-1}(S(n+1))$. Then, $A_n v\in S(n+1)$ and thus 
\[
\sup_{m\ge n+1}\| \cA(m, n+1)A_n v\|=\sup_{m\ge n+1}\|\cA(m, n)v\|<+\infty,
\]
which implies that $\sup_{m\ge n}\|\cA(m, n)v\|<+\infty$, i.e. $v\in S(n)$. The converse inclusion can be established in a similar manner.

Next, we claim that for $m\ge n$,
\[
X=\cA(m, n)X+S(m).
\]
One can easily see that it is sufficient to consider the case when $n=0$. Then, if $m=0$, there is nothing to show since $\cA(m, n)X=X$. Let us now consider the case when $m>0$. Take $v\in X$ and define $\mathbf y=(y_k)_{k\in \mathbb N}$ by $y_m=v$ and $y_k=0$ for $k\neq m$. Then, $\mathbf y\in \ell_{\beta, 0}^1$. Consequently, there exists $\mathbf x=(x_k)_{k\in \mathbb N}\in \ell_\beta^\infty$ such that~\eqref{adm} holds. In particular,
\[
x_m-A_{m-1}x_{m-1}=v 
\]
and 
\[
x_k=A_{k-1}x_{k-1}, \quad \text{for } k\neq m.
\]
Hence, since $\mathbf x\in \ell_\beta^\infty$, we get that $x_m\in S(m)$. On the other hand, $A_{m-1}x_{m-1}=\cA(m, 0)x_0$, yielding that 
\[
v=x_m-A_{m-1}x_{m-1}\in S(m) + \cA(m, 0)X. 
\]
Thus, the desired claim holds.

We proceed by noting that each $S(m)$ is an image of a Banach space under the action of a bounded linear operator. Indeed, let $\mathcal C(m)$ denote the space of all sequences $\mathbf x=(x_n)_{n\ge m}\subset X$ such that 
\[
\|\mathbf x\|_{\mathcal C(m)}:=\sup_{n\ge m} \|x_n\|<+\infty.
\]
Then, $(\mathcal C(m), \| \cdot \|_{\mathcal C(m)})$ is a Banach space. Set $\mathcal C'(m)$ to be the set of all sequences $\mathbf x=(x_n)_{n\ge m}\in \mathcal C(m)$ satisfying
\[
x_{n+1}=A_n x_n, \quad n\ge m.
\]
It is straightforward to verify that $\mathcal C'(m)$ is a closed subspace of $\mathcal C(m)$ and therefore also a Banach space. Then, we observe that $S(m)=\Phi(\mathcal C'(m))$, where $\Phi \colon \mathcal C(m)\to X$ is a bounded linear operator given by 
\[
\Phi(\mathbf x)=x_m, \quad \mathbf x=(x_n)_{n\ge m}\in \mathcal C(m).
\]
It follows now from~\cite[Lemma 3.4.]{Sch} that $S(0)=S_0(0)$ is closed and complemented in $X$. Therefore, there exists a closed subspace $Z\subset X$ such that
\[
X=S(0)\oplus Z.
\]
Take an arbitrary $\mathbf y=(y_n)_{n\in \mathbb Z}\in \ell_{\beta, 0}^\infty$. By our assumption, there exists $\mathbf x=(x_n)_{n\in \mathbb N}\in \ell_{\beta}^\infty$ such that~\eqref{adm} holds. Write $x_0=v_1+v_2$ with $v_1\in S(0)$ and $v_2\in Z$. Set
\[
\tilde x_n:=x_n-\cA(n, 0)v_1.
\]
Since $v_1\in S(0)=S_0(0)=S_\beta(0)$, we easily get that $\tilde{\mathbf x}=(\tilde x_n)_{n\in \mathbb N}\in \ell_\beta^\infty$. Moreover, $\tilde x_0=v_2\in Z$. Therefore, $\tilde{\mathbf x}\in \ell_{\beta, Z}^\infty$ and it is straightforward to verify that the pair $(\tilde{\mathbf x}, \mathbf y)$ satisfies~\eqref{adm}. In addition, let us suppose that the pair $(\overline{\mathbf x}, \mathbf y)$ satisfies~\eqref{adm} with $\overline{\mathbf x}=(\overline{x}_n)_{n\in \mathbb N}\in \ell_{\beta, Z}^\infty$. Then, we have that $\overline{x}_0-\tilde x_0\in S(0)\cap Z$, which gives that $\overline{x}_0-\tilde x_0=0$. Therefore, $\overline{\mathbf x}=\tilde{\mathbf x}$. We have thus proved that the first assumption in the statement of Theorem~\ref{theo: adm to dich} holds. Similarly, one can establish the second assumption in Theorem~\ref{theo: adm to dich} (for this it is sufficient to observe that  for $v\in S(0)$, $n\mapsto \cA(n, 0)v\in \ell_{-\beta}^\infty$).  The desired conclusion now follows from Theorem~\ref{theo: adm to dich}.
\end{proof}
\begin{remark}
\begin{itemize}
\item It is easy to verify that under the assumptions of Theorem~\ref{theo: from dich to adm} we have that $S_0(0)=S_\beta(0)$, making this  a reasonable assumption in the statement of Theorem~\ref{new};
\item in contrast to Theorem~\ref{theo: adm to dich}, the admissibility assumptions in Theorem~\ref{new} do not require uniqueness.
\end{itemize}
\end{remark}

\subsection{Persistence of $(\mu,\nu)$-dichotomy} In this section, as a consequence of the characterization of $(\mu,\nu)$-dichotomy given in Section \ref{sec: charct one sided}, we are going to show that the notion of $(\mu,\nu)$-dichotomy persists under small linear perturbations. 

\begin{theorem}\label{theo: persistance unif bound}
Let $(\gamma_n)_{n\in \bbN}$ be a sequence of numbers with $\gamma_n>0$ for every $n\in \bbN$ such that 
\begin{equation}\label{eq: gamman is summable}
    \sum_{n=0}^{\infty}\gamma_n<+\infty.
\end{equation}
Suppose that \eqref{LE} admits a $(\mu,\nu)$-dichotomy with $\lambda >0$ and condition \eqref{munu} is satisfied and take $\beta \in (0,\lambda-\varepsilon)$. Moreover, let $(B_n)_{n\in \bbN}$ be a sequence of operators in $\mathcal{B}(X)$ with the property that there exists $c>0$ such that 
    \begin{equation}\label{eq: unif condition on Bn}
        \|B_n\|\le \frac{c\gamma_n\mu_{n}^\beta}{\nu_{n+1}\mu_{n+1}^\beta}, \quad \text{for $n\in \mathbb N$.}
    \end{equation}
 Then, if $c$ is sufficiently  small  we have that the nonautonomous difference equation
    \begin{equation}\label{eq: perturbed eq}
        x_{n+1}=(A_n+B_n)x_n, \quad n\in \bbN
    \end{equation}
also admits a $(\mu,\nu)$-dichotomy.
\end{theorem}

\begin{remark}\label{remark: cond on Bn}
Observe that in the case when $\mu_{n+1}\leq K\mu_{n}$ for every $n\in \bbN$ and some $K>0$, that is, the sequence $(\mu_n)_{n\in \bbN}$ does not grow faster than exponential, condition \eqref{eq: unif condition on Bn} can be reduced to 
\begin{equation}\label{eq: unif cond Bn remark}
\|B_n\|\le \frac{c\gamma_n}{\nu_{n+1}}, 
\end{equation}
for $n\in \bbN$ and $c>0$ small enough. In particular, this comment applies to the classical settings of exponential, polynomial and logarithmic dichotomy. Observe moreover that whenever we are in the case of an \emph{uniform} dichotomy, meaning that the sequence $(\nu_n)_{n\in \bbN}$ is constant, condition \eqref{eq: unif cond Bn remark} is basically saying that the classical $\ell^1$-norm of $(\|B_n\|)_{n\in \bbN}$ is small. This indicates that our robustness is, in general, not optimal  as, it is  for example well-known,  that the notion of uniform exponentially dichotomy persists under the requirement that $\sup_n \|B_n\|$ is small. On the other hand, in the case of polynomial and logarithmic dichotomy it does generalize existing results \cite{D1, DSS, DSS2, Silva} since, for instance, we do not impose any bounded growth condition in \eqref{LE} contrary to what is done in those works. 

We note that the robustness of $(\mu, \nu)$-dichotomies was discussed in~\cite[Theorem 3.9]{Chu}. There are important differences between this result and our Theorem~\ref{theo: persistance unif bound}. Firstly, in~\cite{Chu} only the case of invertible dynamics is discussed. Moreover, \cite[Theorem 3.9]{Chu} is not applicable in the case where $\nu_n=1$ for each $n$, since in this setting~\cite[(2.5)]{Chu} is not fulfilled. In particular, it is not applicable to the classical settings of uniform exponential, polynomial and logarithmic dichotomy.

Finally, we would like to compare Theorem~\ref{theo: persistance unif bound} with an unpublished work~\cite{BS3}. Letting $a_n=\mu_n^\lambda$ and $b_n=\mu_n^{-\lambda}$, we observe that~\cite[Corollary 6]{BS3} gives the same conclusion as Theorem~\ref{theo: persistance unif bound} under a slightly stronger condition:
\begin{equation}\label{robalt}
 \|B_n\|\le \frac{c\gamma_n\mu_{n}^\lambda}{\nu_{n+1}\mu_{n+1}^\lambda}, \quad \text{for $n\in \mathbb N$,}
\end{equation}
where $c>0$ is sufficiently small and $(\gamma_n)$ satisfying~\eqref{eq: gamman is summable}. Indeed, observe that since $\beta<\lambda$ and $\mu_n/\mu_{n+1}<1$ one has 
\[
\left (\frac{\mu_n}{\mu_{n+1}}\right )^\lambda \le \left (\frac{\mu_n}{\mu_{n+1}}\right )^\beta.
\]
Consequently, \eqref{robalt}  implies~\eqref{eq: unif condition on Bn}, while the converse is not true (if $(\mu_n)_n$ is not slowly varying).
Therefore, our result is slightly stronger than~\cite[Corollary 6]{BS3}. We stress that the arguments in~\cite{BS3} and the present paper are completely different. In particular, the results in~\cite{BS3} do not rely on admissibility.

\end{remark}

In order to prove Theorem \ref{theo: persistance unif bound} let us introduce some terminology. Given a closed subspace $Z\subset X$ and  $\beta>0$, let us consider consider the linear operators $\mathbf A_{\beta}\colon \mathcal{D}(\mathbf A_{\beta})\subset \ell^\infty_{\beta, Z} \to \ell^1_{\beta ,0}$ and $\mathbf A_{-\beta}\colon \mathcal{D}(\mathbf A_{-\beta})\subset \ell^\infty_{-\beta, Z} \to \ell^1_{-\beta ,0}$ given by
\begin{equation*}
    (\mathbf A_j \mathbf x)_n=x_n-A_{n-1}x_{n-1} \text{ for } n\geq 1
\end{equation*}
and $(\mathbf A_j \mathbf x)_0=0$ for $j=\beta, -\beta$ where $\mathcal{D}(\mathbf A_{\beta})=\{\mathbf x \in \ell^\infty_{\beta ,Z}: \mathbf A_{\beta} \mathbf x\in \ell^1_{\beta ,0}\}$ and $\mathcal{D}(\mathbf A_{-\beta})=\{\mathbf x \in \ell^\infty_{-\beta ,Z}: \mathbf A_{-\beta} \mathbf x\in \ell^1_{-\beta ,0}\}$. Then, using this terminology we can reformulate Theorems \ref{theo: from dich to adm} and \ref{theo: adm to dich} as follows.

\begin{theorem}\label{theo: reformulation charct dich} 
Let us consider the following conditions: 
\begin{itemize}
    \item[i)] Eq.~\eqref{LE} admits a $(\mu,\nu)$-dichotomy.
    \item[ii)] There exist a closed subspace $Z\subset X$ and  $\beta>0$ such that for every $\mathbf y^1\in \ell^1_{\beta ,0}$ there exists a unique $\mathbf x^1\in \ell^\infty_{\beta ,Z}$ satisfying $\mathbf A_{\beta} \mathbf x^1=\mathbf y^1$ and for every $\mathbf y^2\in \ell^1_{-\beta ,0}$ there exists a unique $\mathbf x^2\in \ell^\infty_{-\beta ,Z}$ satisfying $\mathbf A_{-\beta} \mathbf x^2=\mathbf y^2$. In other words, operators $\mathbf A_{\beta}$ and $\mathbf A_{-\beta}$ are invertible.
\end{itemize}
Thus, if $i)$ holds and \eqref{munu} is satisfied then $ii)$ also holds. Reciprocally, if $ii)$ holds then $i)$ also holds.
\end{theorem}

\begin{proof}[Proof of Theorem \ref{theo: persistance unif bound}] 
Let $Z\subset X$ and $\beta>0$ be given by Theorem \ref{theo: reformulation charct dich}. From Theorem \ref{theo: from dich to adm} it follows that $\beta$ may be any value in $(0,\lambda -\varepsilon)$. In particular, we may assume without loss of generality that $\beta$ is such that \eqref{eq: unif condition on Bn} is satisfied. Let us also consider the operators $\mathbf A_{\beta}$ and $\mathbf A_{-\beta}$ given above. We endow $\mathcal{D}(\mathbf A_{\beta})$ and $\mathcal{D}(\mathbf A_{-\beta})$ with the graph norms 
\[\|\mathbf x\|_{\mathbf A_{\beta}}= \|\mathbf x\|_{\infty,\beta}+\|\mathbf A_{\beta}(\mathbf x)\|_{1,\beta} \]
and 
\[\|\mathbf x\|_{\mathbf A_{-\beta}}= \|\mathbf x\|_{\infty,-\beta}+\|\mathbf A_{-\beta}(\mathbf x)\|_{1,-\beta}, \]
respectively. By proceeding as in the proof of Lemma \ref{lem: projections are bounded} we can conclude that $\mathbf A_{\beta}$ and $\mathbf A_{-\beta}$ are closed operators. In particular, $(\mathcal{D}(\mathbf A_j),\|\cdot\|_{\mathbf A_j})$, $j=\beta,-\beta$, are Banach spaces. Consider also the operators $\mathbf B_{\beta}\colon  \ell^\infty_{\beta,Z} \to \ell^\infty_{\beta,0}$ and $\mathbf B_{-\beta}\colon  \ell^\infty_{-\beta,Z} \to \ell^\infty_{-\beta,0}$ given by 
\[(\mathbf B_j\mathbf x)_n=B_{n-1}x_{n-1} \text{ for }n\geq 1\]
and $(\mathbf B_j\mathbf x)_0=0$ for $j=\beta, -\beta$. Observe that condition \eqref{eq: unif condition on Bn} guarantee that both $\mathbf B_{\beta}$ and $\mathbf B_{-\beta}$ are well-defined. 

Now, using \eqref{eq: unif condition on Bn} we get that for every $\mathbf x\in \mathcal{D}(\mathbf A_{\beta})$, 
\begin{equation}\label{eq: est0 pers}
\begin{split}
\|((\mathbf A_{\beta}-\mathbf B_{\beta})\mathbf x)_n\|&\leq \|(\mathbf B_{\beta}\mathbf x)_n\|+\|(\mathbf A_{\beta}\mathbf x)_n\|\\
&\leq \frac{c\gamma_{n-1}\mu_{n-1}^\beta}{\nu_{n}\mu_{n}^\beta} \|x_{n-1}\|+\|(\mathbf A_{\beta}\mathbf x)_n\|
\end{split}
\end{equation}
for every $n\geq 1$. Consequently, using \eqref{eq: gamman is summable} and that $\mathbf x \in  \ell^\infty_{\beta}$ and $\mathbf A_{\beta}\mathbf x \in  \ell^1_{\beta,0}$ we get that 
\begin{equation}\label{eq: est1 pers}
\begin{split}
\|(\mathbf A_{\beta}-\mathbf B_{\beta})\mathbf x\|_{1,\beta}&\leq \sum_{n=0}^{\infty} c\gamma_{n} \mu_n^{\beta} \|x_{n}\|+\|\mathbf A_{\beta}\mathbf x\|_{1,\beta}\\
&\leq c\|\mathbf x\|_{\infty, \beta} \sum_{n=0}^{\infty} \gamma_{n}  +\|\mathbf A_{\beta}\mathbf x\|_{1,\beta} <+\infty
\end{split}
\end{equation}
and $(\mathbf A_{\beta}-\mathbf B_{\beta})\mathbf x\in \ell^1_{\beta,0}$. In particular, the operator $\mathbf A_{\beta}-\mathbf B_{\beta}\colon (\mathcal{D}(\mathbf A_{\beta}),\|\cdot\|_{\mathbf A_{\beta}}) \to \ell^1_{\beta,0}$ given by $(\mathbf A_{\beta}-\mathbf B_{\beta})\mathbf x$ is well-defined. Moreover, by \eqref{eq: est1 pers} we have that
\[\|(\mathbf A_{\beta}-\mathbf B_{\beta})\mathbf x\|_{1,\beta}\leq K\|\mathbf x\|_{\mathbf A_{\beta}}\]
for some $K>0$ and, consequently, $\mathbf A_{\beta}-\mathbf B_{\beta}$ is bounded. Furthermore, using part of the estimate obtained in \eqref{eq: est0 pers} we get that
\[
\begin{split}
\|\mathbf A_{\beta}\mathbf x -(\mathbf A_{\beta}-\mathbf B_{\beta})\mathbf x\|_{1,\beta}\leq c\sum_{n=0}^{\infty} \gamma_{n}  \|\mathbf x\|_{\infty, \beta} \leq c\sum_{n=0}^{\infty} \gamma_{n}  \|\mathbf x\|_{\mathbf A_{\beta}} .
\end{split}
\]
Therefore, since $\mathbf A_{\beta}$ is invertible (recall Theorem \ref{theo: reformulation charct dich}), we conclude that for $c>0$ small enough $\mathbf A_{\beta}-\mathbf B_{\beta}$ is also invertible. Proceeding in a similar manner we conclude that $\mathbf A_{-\beta}-\mathbf B_{-\beta}\colon (\mathcal{D}(\mathbf A_{-\beta}),\|\cdot\|_{\mathbf A_{-\beta}}) \to \ell^1_{-\beta,0}$ is also a well-defined bounded and invertible linear operator. Then, combining these two observations with Theorem \ref{theo: reformulation charct dich} we get that \eqref{eq: perturbed eq} admits a $(\mu, \nu)$-dichotomy.
\end{proof}

We now present an example to which our results are applicable to but those of \cite{Silva} are not. Observe that in this example condition \cite[(10)]{Silva} is not satisfied.

\begin{example}
Take $X=\mathbb R$, and consider  any growth rate $\mu=(\mu_n)_{n\in \mathbb N}$  such that $\lim_{n\to \infty}\frac{\mu_{n+1}}{\mu_n}=\infty$. For  example, we can take 
 $\mu_n=e^{e^n}$, $n\in \mathbb N$. Moreover, let 
\[
A_n=\left (\frac{\mu_{n+1}}{\mu_n}\right )^{-\frac 1 2 }, \quad n\in \mathbb N.
\]
Observe that 
\[
\mathcal A(m, n)=\left (\frac{\mu_m}{\mu_n}\right )^{-\frac 1 2} \quad \text{for $m\ge n$.}
\]
Consequently, the sequence $(A_n)_{n\in \mathbb N}$ admits a $(\mu, \nu)$-dichotomy where $\nu_n=1$ for $n\in \mathbb N$ with projections  $P_n=\Id$ for $n\in \mathbb N$.  Consider the sequence $\mathbf y=(y_n)_{n\in \mathbb N}\subset \mathbb R$ given by $y_0=0$ and $y_n=1$ for $n\ge 1$. Clearly, $\mathbf y$ is bounded. We claim that there does not exist a bounded sequence $\mathbf x=(x_n)_{n\in \mathbb N}\subset \mathbb R$  with $x_0=0$ such that 
\begin{equation*}
\phi_{n+1}^\mu(x_{n+1}-A_n x_n)=y_{n+1}, \quad n\in \mathbb N,
\end{equation*}
where $\phi_n^\mu:=\frac{\mu_n}{\mu_{n+1}-\mu_n}$.
Indeed, assuming such a sequence does exist, one can easily verify that 
\[
x_n=\sum_{k=1}^n \frac{1}{\phi_k^\mu} \mathcal A(n, k)y_k=\sum_{k=1}^n \frac{1}{\phi_k^\mu} \left (\frac{\mu_n}{\mu_k}\right )^{-\frac  12 }=\mu_n^{-\frac 1 2}\sum_{k=1}^n \frac{1}{\phi_k^\mu} \mu_k^{\frac 1 2}, \quad n\ge 1.
\]
On the other hand, by using the first inequality in~\cite[(7)]{DSV} (whose proof does not rely on the slow-growth property~\cite[p.5 (iii)]{DSV}), we have that 
\[
\sum_{k=1}^n \frac{1}{\phi_k^\mu} \mu_k^{\frac 1 2}=\sum_{k=1}^n \mu_k^{-\frac 1 2}(\mu_{k+1}-\mu_k)\ge 2(\mu_{n+1}^{\frac 1 2}-\mu_1^{\frac 1 2}),
\]
and consequently
\[
x_n\ge 2 \left ( (\mu_{n+1}/\mu_n)^{\frac 12 }-(\mu_1/\mu_n)^{\frac 12 }\right ), \quad n\ge 1.
\]
Since $\lim_{n\to \infty}\frac{\mu_{n+1}}{\mu_n}=\infty$, we obtain a contradiction. This shows that~\cite[Theorem 3.3]{Silva} does not hold in this case (see \cite[Remark 2]{Silva}).
\end{example}

\begin{remark}
We note that the example given in~\cite[Example 1]{D1} illustrates that the bounded growth condition~\cite[(14)]{Silva} cannot be omitted as an assumption in~\cite[Theorem 3.4]{Silva}. However, our results do not require such an assumption.
\end{remark}

\section{The case of two-sided dynamics}\label{sec: two sided}

In this section we are going to consider the case of two-sided dynamics, that is, the case when $I=\bbZ$. 

\subsection{Characterization of $(\mu,\nu)$-dichotomy} \label{sec: charct two sided}
Following the ideas of Section \ref{sec: charct one sided}, we will now present a characterization of $(\mu ,\nu)$-dichotomy in terms of the admissibility of certain weighted spaces. In the present context, in addition to the admissibility of spaces of the form $\ell^1_{\beta}$ and $\ell^\infty_{\beta}$ for some appropriate values of $\beta\in \bbR$, we will also have to consider the admissibility of spaces $\ell^1_{\beta,|\cdot|}$ and $\ell^\infty_{\beta,|\cdot|}$.

\begin{theorem}\label{theo: from dich to adm - two sided}
Suppose that $(A_n)_{n\in \mathbb Z}$ admits a $(\mu,\nu)$-dichotomy with respect to projections $P_n$ and let $\lambda>0$ be such that~\eqref{eq: Es discrete} and~\eqref{eq: Eu discrete} hold. Moreover, suppose that there exists $\varepsilon \in [0, \lambda)$ such that 
\begin{equation}\label{munu two sided}
\sup_{n\in \mathbb N}(\mu_{-n}^{\varepsilon}\nu_{-n})<+\infty \; \text{ and } \;\sup_{n\in \mathbb N}(\mu_n^{-\varepsilon}\nu_n)<+\infty.
\end{equation}
Then,  for each $\beta \in (-(\lambda-\varepsilon), \lambda-\varepsilon)$ the following holds:
\begin{enumerate}
    \item for every $\mathbf y=(y_n)_{n\in \mathbb Z}\in \ell^1_{\beta}$, there exists a unique $\mathbf x=(x_n)_{n\in \mathbb N}\in \ell_{\beta}^\infty$ such that
\begin{equation}\label{adm two sided}
x_{n+1}-A_n x_n=y_{n+1} \quad n\in \mathbb Z.
\end{equation}
    \item for every $\mathbf y=(y_n)_{n\in \mathbb Z}\in \ell^1_{\beta,|\cdot |}$, there exists a unique $\mathbf x=(x_n)_{n\in \mathbb N}\in \ell_{\beta,|\cdot |}^\infty$ such that \eqref{adm two sided} holds.
\end{enumerate}
\end{theorem}

\begin{proof}
  Take $\beta \in (-(\lambda-\varepsilon), \lambda-\varepsilon)$, $\mathbf y=(y_n)_{n\in \mathbb Z}\in \ell^1_{\beta}$  and set
\begin{equation}\label{xn two sided}
x_n:=\sum_{k=-\infty}^n \mathcal A(n, k)P_k y_k-\sum_{k=n+1}^\infty \mathcal A(n, k)(\Id-P_k)y_k, \quad n\in \mathbb Z.
\end{equation}
Then, proceeding as in the proof of Theorem \ref{theo: from dich to adm} we obtain that $\mathbf x=(x_n)_{n\in \mathbb Z}\in \ell^\infty_{\beta}$ and that \eqref{adm two sided} is satisfied.
We now establish the uniqueness of such $\mathbf x$. For this purpose, by linearity, all we have to do is to prove that the unique sequence $\mathbf z=(z_n)_{n\in \mathbb Z}$ in $\ell^\infty_{\beta}$ that satisfies
\[
z_{n+1}=A_nz_n, \quad n\in \mathbb Z
\]
is the null sequence. That is, $z_n=0$ for all $n\in \bbZ$. Let us consider
\[z^s_n=P_nz_n \; \text{ and } \; z^u_n=(\Id -P_n)z_n \; \text{ for } n\in \bbZ.\]
Then, $z_n=z^s_n+z^u_n$ and, by \eqref{eq: Pn An commute},
\[z^s_{n+1}=A_nz^s_n \; \text{ and } \; z^u_{n+1}=A_nz^u_n \; \text{ for all } n\in \bbZ.\]
Thus, using \eqref{eq: Es discrete}, we get that for every $m\geq n$,
\[
\begin{split}
\mu_m^\lambda\|z^s_m\|&=\mu_m^\lambda\|\mathcal A(m,n)z^s_n\|\\
&=\mu_m^\lambda\|\mathcal A(m,n)P_nz_n\|\\
&\leq D\nu_n \mu_n^{\lambda}\|z_n\|\\
&\leq D\nu_n \mu_n^{\lambda-\beta} \|\mathbf z\|_{\infty, \beta} \\
&\le D\mu_n^{\lambda-\beta-\varepsilon}\|\mathbf z\|_{\infty, \beta} \cdot \sup_{k\leq m}(\mu_k^{\varepsilon}\nu_k).
\end{split}
\]
Consequently, since $\lim_{n\to -\infty}\mu_n=0$, letting $n\to -\infty$ and using~\eqref{munu two sided} and $\beta<\lambda-\varepsilon$ it follows that $z^s_m=0$ for every $m\in \bbZ$. Similarly, using \eqref{eq: Eu discrete}, we get that for every $m< n$,
\[
\begin{split}
\mu_m^{-\lambda}\|z^u_m\|&=\mu_m^{-\lambda} \|\mathcal A(m,n)z^u_n\|\\
&=\mu_m^{-\lambda}\|\mathcal A(m,n)(\Id-P_n)z_n\|\\
&\leq D\nu_n \mu_n^{-\lambda}\|z_n\|\\
&\leq D\nu_n \mu_n^{-(\lambda+\beta)} \|\mathbf z\|_{\infty, \beta} \\
&\le D\mu_n^{-(\lambda+\beta-\varepsilon)}\|\mathbf z\|_{\infty, \beta} \cdot \sup_{k\geq  m}(\mu_k^{-\varepsilon}\nu_k).
\end{split}
\]
Therefore, since $\lim_{n\to +\infty}\mu_n=+\infty$, letting $n\to +\infty$ and using~\eqref{munu two sided} and $\beta>-(\lambda-\varepsilon)$ it follows that $z^u_m=0$ for every $m\in \bbZ$. Combining these observations we conclude that $z_m=z^s_m+z^u_m=0$ for every $m\in \bbZ$ completing the proof of the first claim in the theorem. Let us now prove the second one.

Given  $\beta \in (-(\lambda-\varepsilon), \lambda-\varepsilon)$ and  $\mathbf y=(y_n)_{n\in \mathbb Z}\in \ell^1_{\beta, |\cdot|}$, let us consider $x_n$, $n\in \bbZ$, defined by~\eqref{xn two sided}. Then, using \eqref{eq: Es discrete} and \eqref{eq: Eu discrete}, if $n\geq n_0$ we get that 
\[
\begin{split}
\|x_n\| &\le D \sum_{k=-\infty}^n \left (\frac{\mu_n}{\mu_k}\right )^{-\lambda}\nu_k \|y_k\|+D \sum_{k=n+1}^\infty \left (\frac{\mu_k}{\mu_n}\right )^{-\lambda}\nu_k \|y_k\| \\
&= D \sum_{k=-\infty}^{n_0-1} \left (\frac{\mu_n}{\mu_k}\right )^{-\lambda}\nu_k \|y_k\|+D \sum_{k=n_0}^n \left (\frac{\mu_n}{\mu_k}\right )^{-\lambda}\nu_k \|y_k\|\\
&\phantom{=}+D \sum_{k=n+1}^\infty \left (\frac{\mu_k}{\mu_n}\right )^{-\lambda}\nu_k \|y_k\| \\
&= D\sum_{k=-\infty}^{n_0-1} \left (\frac{\mu_n}{\mu_k}\right )^{-(\lambda-|\beta|)} 
\mu_n^{-|\beta|} \mu_k^{|\beta|} \nu_k \|y_k\|\\
&\phantom{=}+D \sum_{k=n_0}^n \left (\frac{\mu_n}{\mu_k}\right )^{-(\lambda+|\beta|)}  \mu_n^{|\beta|}\mu_k^{-|\beta|} \nu_k \|y_k\|\\
&\phantom{=} +D\sum_{k=n+1}^\infty \left (\frac{\mu_k}{\mu_n}\right )^{-(\lambda-|\beta|)} \mu_n^{|\beta|}\mu_k^{-|\beta|} \nu_k \|y_k \|. \\
\end{split}
\]
Thus, since $(\mu_k)_{k\in \bbZ}$ is strictly increasing, $|\beta|<\lambda$ and $\mu_n\geq 1$ for every $n\geq n_0$, it follows that
\[
\begin{split}
\|x_n\| & \leq D\mu_n^{|\beta|} \left( \sum_{k=-\infty}^{n_0-1} \mu_k^{|\beta|} \nu_k \|y_k\|+\sum_{k=n_0}^\infty \mu_k^{-|\beta|} \nu_k \|y_k \|\right)\\
\end{split}
\]
which implies that
\[
\mu_n^{-|\beta|}\|x_n\| \leq  D\|\textbf{y}\|_{1,\beta,|\cdot |}.
\]
Similarly, in the case when $n< n_0$ we can prove that
\[\mu_n^{|\beta|} \|x_n\| \leq D\|\textbf{y}\|_{1,\beta,|\cdot |}.\]
Consequently, $\textbf{x}=(x_n)_{n\in \bbZ}\in \ell^\infty_{\beta,|\cdot|}$. Moreover, proceeding again as in the proof of Theorem \ref{theo: from dich to adm} we conclude that \eqref{adm two sided} is satisfied. Finally, the uniqueness of this $\mathbf x\in \ell^\infty_{\beta, |\cdot|}$ can be obtained via an argument similar to the one we did in the proof of the first claim of the theorem. The proof of the theorem is completed.
\end{proof}

Our next theorem gives us the converse result of Theorem \ref{theo: from dich to adm - two sided}.

\begin{theorem}\label{theo: adm to dich - two sided}
Suppose that there exists $\beta>0$  such that:
\begin{enumerate}
\item for each $\mathbf y=(y_n)_{n\in \mathbb Z} \in \ell^1_{\beta}$, there exists a unique $\mathbf x=(x_n)_{n\in \mathbb Z}\in \ell_{\beta}^\infty$ such that~\eqref{adm two sided} holds;
\item for each $\mathbf y=(y_n)_{n\in \mathbb Z}\in \ell^1_{-\beta}$, there exists a unique $\mathbf x=(x_n)_{n\in \mathbb Z}\in \ell_{-\beta}^\infty$ such that~\eqref{adm two sided} holds;
\item for each $\mathbf y=(y_n)_{n\in \mathbb Z}\in \ell^1_{\beta,|\cdot|}$, there exists a unique $\mathbf x=(x_n)_{n\in \mathbb Z}\in \ell_{\beta,|\cdot|}^\infty$ such that~\eqref{adm two sided} holds.
\end{enumerate}
Then, $(A_n)_{n\in \mathbb Z}$ admits a $(\mu,\nu)$-dichotomy. 
\end{theorem}

\begin{proof}
    We will proceed as in the proof of Theorem \ref{theo: adm to dich} and again, as in the aforementioned proof, in all the auxiliary results we assume that the hypotheses of Theorem \ref{theo: adm to dich - two sided} are satisfied even though we do not write it explicitly each time.
 We start with two auxiliary observations.
    \begin{lemma}\label{aux two sided}
Let $\mathbf y=(y_n)_{n\in \mathbb Z}\in \ell^1_{\beta}\cap \ell^1_{-\beta}$, $\mathbf x^1=(x_n^1)_{n\in \mathbb Z}\in \ell^\infty_{\beta}$ and $\mathbf x^2=(x_n^2)_{n\in \mathbb Z}\in \ell^\infty_{-\beta}$ be such that both pairs $(\mathbf x^i, \mathbf y)$, $i\in \{1, 2\}$ satisfy~\eqref{adm two sided}. Then, $x_n^1=x_n^2$ for every $n\in \mathbb Z$. 
\end{lemma}
\begin{proof}[Proof of Lemma \ref{aux two sided}]
    Since $\ell^1_{\beta}\subset \ell^1_{\beta,|\cdot|}$ and $\ell^1_{-\beta}\subset \ell^1_{\beta,|\cdot|}$, it follows that $\mathbf y\in \ell^1_{\beta,|\cdot|}$. Similarly, since $\ell^\infty_{\beta}\subset \ell^\infty_{\beta,|\cdot|}$ and $\ell^\infty_{-\beta}\subset \ell^\infty_{\beta,|\cdot|}$, it follows that $\mathbf x^1, \mathbf x^2\in \ell^\infty_{\beta,|\cdot|}$. Thus, by the uniqueness given in the third assumption of the statement of Theorem \ref{theo: adm to dich - two sided} it follows that $x_n^1=x_n^2$ for every $n\in \mathbb Z$ as claimed.
\end{proof}

 Let us consider $T_{j}\colon \ell^1_{j}\to \ell^\infty_{j}$ given by $T_j(\mathbf y)=\mathbf x$ where $\mathbf x$ is the unique element in $l^\infty_{j}$ such that~\eqref{adm two sided} holds for $j=\beta, -\beta$. 

\begin{lemma}\label{lem: T_beta is bounded}
Then $T_\beta$ and $T_{-\beta}$ are bounded linear operators.
\end{lemma}
\begin{proof}[Proof of Lemma \ref{lem: T_beta is bounded}]
The proof can be obtained by proceeding as in the proof of Lemma \ref{lem: projections are bounded}.
\end{proof}

    Given $n\in \bbZ$, let us consider
\[S(n):=\left\{v\in X: \; \sup_{m\geq n} \|\cA(m,n)v\|<+\infty  \right\}.\]
Similarly, let $U(n)$ be the space of all $v\in X$ for which there exists a sequence $(z_m)_{m\leq n}$ such that  $z_n=v$, $z_m=A_{m-1}z_{m-1} $ for every  $m\leq n$  and $\sup_{m\leq n}\|z_m\|<+\infty $.

It is easy to see that 
\begin{equation*}
    A_nS(n)\subset S(n+1) \quad \text{ and } \quad A_nU(n)\subset U(n+1)
\end{equation*}
for every $n\in \bbZ$. Moreover, we have the following observations.

\begin{lemma}\label{lem: spliting two sided}
For every $n\in \bbZ$, 
\begin{equation}\label{split two sided}
X=S(n)\oplus U(n).
\end{equation}
\end{lemma}
\begin{proof}[Proof of Lemma \ref{lem: spliting two sided}]
Fix $n\in \bbZ$. Given $v\in X$, we define $\mathbf y=(y_m)_{m\in \mathbb Z}$ by $y_n=v$ and $y_m=0$ for $m\neq n$. Then, $\mathbf y\in \ell^1_{\beta} \cap \ell^1_{-\beta}$. By our assumption and Lemma \ref{aux two sided}, there exists $\mathbf x=(x_m)_{m\in \mathbb Z}\in \ell^\infty_{\beta}\cap \ell^\infty_{-\beta}$ such that 
\[
x_{m+1}-A_m x_m=y_{m+1}, \quad m\in \mathbb Z.
\]
In particular, 
\begin{equation}\label{eq: aux spl 1}
x_n-v=A_{n-1}x_{n-1}
\end{equation}
and
\begin{equation}\label{eq: aux spl 2}
x_{m+1}=A_m x_m \quad \text{ for all } m\neq n-1.
\end{equation}
Thus, \eqref{eq: aux spl 2} combined with the fact that $(x_m)_{m\in \bbZ}\in \ell_\beta^\infty $ implies that $x_n \in S(n)$. Similarly, \eqref{eq: aux spl 1} and \eqref{eq: aux spl 2} combined with the fact that $(x_m)_{m\in \bbZ}\in \ell_{-\beta}^\infty $ implies that $x_n-v\in U(n)$. Therefore,
\[v=x_n+(v-x_n)\in S(n)+U(n)\]
and $X=S(n)+U(n)$.

Suppose now that $v\in S(n)\cap U(n)$.  Then there exists $(z_m)_{m\leq n}$ such that $z_n=v$, $z_m=A_{m-1}z_{m-1}$ for every $m\leq n$ and $\sup_{m\leq n} \|z_m\|<+\infty$. Define
\[x_m=\left\{\begin{array}{ll}
  z_m,
  & \text{ for } m\leq n;
\\
 \cA(m,n)v,
 & \text{ for } m> n.
\end{array}\right.\]
Thus, since $v\in S(n)$ and $\sup_{m\leq n}\|z_m\|<+\infty$, it follows that $\sup_{m\in \bbZ}\|x_m\|<+\infty$. In particular, $(x_m)_{m\in \bbZ}\subset \ell^\infty_{\beta,|\cdot|}$ and 
\[x_{m+1}-A_mx_m=0\; \text{ for every }m\in \bbZ.\]
Therefore, from the uniqueness given in the third  hypothesis of Theorem \ref{theo: adm to dich - two sided} we get that $x_m=0$ for every $m\in \bbZ$ which implies that $v=0$ and $S(n)\cap U(n)=\{0\}$. This concludes the proof.
\end{proof}

\begin{lemma}\label{lem: isom two sided}
     For every $n\in \bbZ$, the operator $A_n|_{U(n)}\colon U(n)\to U(n+1)$ is an isomorphism.
\end{lemma}
\begin{proof}[Proof of Lemma \ref{lem: isom two sided}]
    Fix $n\in \bbZ$. Let us start by observing that $A_n|_{U(n)}$ is injective. Suppose there exists $v\in U(n)$ such that $A_nv=0$. By the definition of $U(n)$ there exists $(z_m)_{m\leq n}$ such that $z_n=v$, $z_m=A_{m-1}z_{m-1}$ for every $m\leq n$ and $\sup_{m\leq n} \|z_m\|<+\infty$. Then the sequence $(x_m)_{m\in \bbZ}$ given by 
    \[x_m=\left\{\begin{array}{ll}
     z_m, & \text{ for } m\leq n;\\
 0, & \text{ for } m> n,
\end{array}\right.\]
is in $l^\infty_{\beta }$ and satisfies \eqref{adm two sided} with $\mathbf y =(y_m)_{m\in \bbZ}$ and $y_m=0$ for every $m\in \bbZ$. Consequently, by the uniqueness in the hypothesis of Theorem \ref{theo: adm to dich - two sided}, it follows that $x_m=0$ for every $m\in \bbZ$. In particular, $v=0$ and $A_n|_{U(n)}$ is injective.

Now, given $v\in U(n+1)$, let us consider a sequence $(z_m)_{m\leq n+1}$ such that $z_{n+1}=v$, $z_m=A_{m-1}z_{m-1}$ for every $m\leq n+1$ and $\sup_{m\leq n+1}\|z_m\|<+\infty$. Then,  $z_n\in U(n)$ and $A_nz_n=v$ which proves that  $A_n|_{U(n)}\colon U_n\to U(n+1)$ is surjective.
\end{proof}

\begin{lemma}\label{lem: projections are bounded two sided}
For each $n\in \bbZ$, let $P_n\colon X\to S(n)$ be the projection associated with~\eqref{split two sided}. Then, there exists $D>0$ such that 
\begin{equation}\label{probound two sided}
\|P_n\| \le D\nu_n, \quad n\in \mathbb Z.
\end{equation}
\end{lemma}
\begin{proof}[Proof of Lemma \ref{lem: projections are bounded two sided}]
The proof of this result is the same, mutatis mutandis, as the proof of Lemma \ref{lem: projections are bounded}.
\end{proof}

\begin{lemma}\label{lem: bounds stable direction two sided}
There exists $C>0$ such that 
\[
\|\mathcal A(m, n)v\| \le C\left (\frac{\mu_m}{\mu_n}\right )^{-\beta} \nu_n \|v\|,\]
for every $m\ge n$ and $v\in S(n)$.
\end{lemma}

\begin{proof}[Proof of Lemma \ref{lem: bounds stable direction two sided}]
Let $n\in \bbZ$ and $v\in S(n)$. We define sequences $(y_m)_{m\in \mathbb Z}$ and $(x_m)_{m\in \mathbb Z}$ by 
\[
y_m=\begin{cases}
v & m=n; \\
0 & m\neq n,
\end{cases} \quad \text{and} \quad 
x_m=\begin{cases}
\mathcal A(m, n)v & m\ge n; \\
0& m<n.
\end{cases}
\]
Then, $\mathbf y=(y_m)_{m\in \mathbb N}\in \ell^1_\beta\cap \ell^1_{-\beta}$ and $\mathbf x=(x_m)_{m\in \mathbb Z}\in \ell^\infty_{-\beta}$. Moreover, \eqref{adm two sided} holds. Then, from Lemma~\ref{aux two sided} it follows that $\mathbf x\in \ell^\infty_{\beta}$. Thus, considering $T_{\beta}\colon \ell^1_{\beta}\to \ell^\infty_{\beta}$ as in Lemma \ref{lem: T_beta is bounded} and $C=\|T_\beta\|+1$ we get that
\[
\mu_m^\beta \|x_m\| \le \| \mathbf x\|_{\infty, \beta}= \|T_\beta(\mathbf y)\|_{\infty, \beta}\le C\|y\|_{1, \beta}=C\mu_n^\beta \nu_n \|v\|
\]
for $m\ge n$. Thus,
\[
\|\mathcal A(m, n)v\| \le C\left (\frac{\mu_m}{\mu_n}\right )^{-\beta}\nu_n \|v\|
\]
for every $m\ge n$ as claimed. 
\end{proof}

\begin{lemma}\label{lem: bounds unstable direction two sided}
There exists $\tilde C>0$ such that 
\[
\|\mathcal A(m, n)v\| \le \tilde C \left (\frac{\mu_n}{\mu_m}\right )^{-\beta}\nu_n \|v\|,\]
for every $m\leq  n$ and $v\in U(n)$.
\end{lemma}

\begin{proof}[Proof of Lemma \ref{lem: bounds unstable direction two sided}]     
Take $n\in \bbZ$ and $v\in U(n)$. Then by Lemma \ref{lem: isom two sided} we may consider 
\[
 y_m=\begin{cases}
 -v & m=n; \\
0 &m\neq n,
\end{cases}\quad \text{and} \quad x_m=\begin{cases}
\mathcal A(m, n)v & m<n;\\
0 &m\ge n.
\end{cases} 
\]
Note that $\mathbf y=(y_m)_{m\in \mathbb Z}\in \ell^1_{\beta}\cap \ell^1_{-\beta}$ and $\mathbf x=(x_m)_{m\in \mathbb Z}\in \ell^\infty_{\beta}$. Furthermore, \eqref{adm two sided} holds. Thus, by Lemma \ref{aux two sided} it follows that $\mathbf x\in \ell^\infty_{-\beta}$. Then, considering $T_{-\beta}\colon \ell^1_{-\beta}\to \ell^\infty_{-\beta}$ given in Lemma \ref{lem: T_beta is bounded} and taking $\tilde C=\| T_{-\beta}\|+1$ we get that
\[
\mu_m^{-\beta}\| x_m\| \leq \| \mathbf x\|_{\infty, -\beta}= \| T_{-\beta}(\mathbf y)\|_{\infty, -\beta}\le \tilde C\|y\|_{1, -\beta}=\tilde C\mu_n^{-\beta} \nu_n \|v\|,
\]
for every $m<n$. This easily implies that 
\[
\| \mathcal A(m,n)v\| \le \tilde C \left (\frac{\mu_n}{\mu_m}\right )^{-\beta}\nu_n \|v\|,
\]
for every $m<n$ as claimed.
\end{proof}
For $m, n\in \mathbb Z$, set
\[
\mathcal G(m, n):=\begin{cases}
\mathcal A(m, n)P_n & m\ge n; \\
-\mathcal A(m, n)(\Id-P_n) &m<n.
\end{cases}
\]
Then, by Lemmas~\ref{lem: projections are bounded two sided}, \ref{lem: bounds stable direction two sided} and~\ref{lem: bounds unstable direction two sided} there exists $C'>0$ such that 
\begin{equation}\label{G two sided}
\|\mathcal G(m, n)\| \le C'\begin{cases}
\left (\frac{\mu_m}{\mu_n}\right )^{-\beta}\nu_n^2 & m\ge n; \\
\left (\frac{\mu_n}{\mu_m}\right )^{-\beta}\nu_n^2
& m<n.
\end{cases}
\end{equation}
\begin{lemma}\label{Green two sided}
There exists $C''>0$ such that 
\[
\|\mathcal G(m, n)\| \le C''\begin{cases}
\left (\frac{\mu_m}{\mu_n}\right )^{-\beta}\nu_n & m\ge n; \\
\left (\frac{\mu_n}{\mu_m}\right )^{-\beta}\nu_n
& m<n.
\end{cases}
\]
\end{lemma}
\begin{proof}[Proof of the Lemma~\ref{Green two sided}]
Take $n\in \bbZ$ and $v\in X$. We define a sequence $\mathbf y=(y_m)_{m\in \mathbb Z}$ by $y_n=v$ and $y_m=0$ for $m\neq n$. Then, $\mathbf y\in \ell_{\beta}^1\cap \ell_{-\beta}^1$. Set
\[
x_m:=\mathcal G(m, n)v, \quad m\in \mathbb Z.
\]
It is easy to verify that~\eqref{adm two sided} holds.
We claim that $\mathbf x=(x_m)_{m\in \mathbb Z}\in \ell_{\beta}^\infty$. Indeed, by \eqref{G two sided}, for $m\ge n$ we have that 
\[
\|x_m\|=\|\mathcal G(m, n)v\|\le C'\left (\frac{\mu_m}{\mu_n}\right )^{-\beta}\nu_n^2\|v\|,
\]
which implies that 
\[
\sup_{m\geq n}(\mu_m^\beta \|x_m\|)<+\infty.
\]
Similarly, for $m< n$ we have that 
\[
\|x_m\|=\|\mathcal G(m, n)v\|\le C'\left (\frac{\mu_n}{\mu_m}\right )^{-\beta}\nu_n^2\|v\|,
\]
which implies that
\[
\mu_m^{-\beta} \|x_m\|\leq C'\mu_n^{-\beta}\nu_n^2\|v\|.
\]
Thus, since $\lim_{m\to -\infty}\mu_m=0$ and, in particular, $\mu_m^{\beta}\le \mu_m^{-\beta}$ for $m$ sufficiently small, the previous inequality implies that
\[
\sup_{m<n}(\mu_m^\beta \|x_m\|)<+\infty.
\]
Combining these observations we conclude that $\mathbf x=(x_m)_{m\in \mathbb Z}\in \ell_{\beta}^\infty$ and $T_\beta (\mathbf y)=\mathbf x$, where $T_\beta$ is as in Lemma~\ref{lem: T_beta is bounded}. Therefore, for $m\ge n$ we have that 
\[
\mu_m^\beta \|\mathcal G(m, n)v\|=\mu_m^\beta \|x_m\|\le \|\mathbf x\|_{\infty, \beta} \le \|T_{\beta}\| \cdot \|\mathbf y\|_{1, \beta}=\|T_\beta \| \mu_n^\beta \nu_n \|v\|,
\]
and thus
\[
\| \mathcal G(m, n)\| \le \|T_\beta\| \left(\frac{\mu_m}{\mu_n}\right)^{-\beta}\nu_n, \quad m\ge n.
\]
Similarly, one can treat the case when $m<n$.
\end{proof}

Now the proof of Theorem \ref{theo: adm to dich - two sided} can be easily completed by combining the previous auxiliary results. 
\end{proof}

\subsection{Persistence of $(\mu,\nu)$-dichotomy} We now present a version of Theorem \ref{theo: persistance unif bound} in the case of two-sided dynamics.

\begin{theorem}\label{theo: persistance unif bound two sided}
Let $(\gamma_n)_{n\in \bbZ}$ be a sequence of numbers with $\gamma_n>0$ for every $n\in \bbZ$ such that 
\begin{equation*}\label{eq: gamman is summable two sided}
    \sum_{n\in \bbZ}\gamma_n<+\infty.
\end{equation*}
Suppose that \eqref{LE} admits a $(\mu,\nu)$-dichotomy with $\lambda>0$ and condition \eqref{munu two sided} is satisfied and take $\beta \in (0,\lambda-\varepsilon)$. Moreover, let $(B_n)_{n\in \bbZ}$ be a sequence of operators in $\mathcal{B}(X)$ with the property that there exists $c>0$ such that 
    \begin{equation}\label{eq: unif condition on Bn two sided}
        \|B_n\|\le \frac{c\gamma_n\mu_{n}^\beta}{\nu_{n+1}\mu_{n+1}^\beta}, \quad \text{for $n\in \mathbb Z$.}
    \end{equation}
 Then, if $c$ is small enough we have that the nonautonomous difference equation
    \begin{equation*}\label{eq: perturbed eq two sided}
        x_{n+1}=(A_n+B_n)x_n, \quad n\in \bbZ
    \end{equation*}
also admits a $(\mu,\nu)$-dichotomy.
\end{theorem}
\begin{proof}
    The proof of this result is similar to the proof of Theorem \ref{theo: persistance unif bound} and, therefore, we refrain from writing it.
\end{proof}
\begin{remark}
    A comment similar to Remark \ref{remark: cond on Bn} also applies to condition \eqref{eq: unif condition on Bn two sided}.
\end{remark}


\medskip{\bf Acknowledgements.}
We would like to thank the referees for their useful comments that helped us to improve our paper and to C. Silva for useful discussions related to~\cite{Silva}. D.D is also grateful to W. Zhang and L. Zhou for ideas related to~\cite{DZZ, DZZ1} which inspired results of the present paper.
L.~Backes was partially supported by a CNPq-Brazil PQ fellowship under Grant No. 307633/2021-7.
D.~Dragi\v cevi\' c was supported in part by University of Rijeka under the project uniri-iskusni-prirod-23-98
3046.

\vspace{0.1in}
\medskip{\bf Statements and Declarations}
\vspace{0.1in}

\textbf{Competing Interests:} no potential conflict of interest was reported by the authors.


\end{document}